\documentclass{amsart}
\usepackage[utf8]{inputenc}

\usepackage{amsthm, amsmath, amsfonts, amssymb}
\usepackage{mathtools}

\usepackage{enumerate}
\usepackage{hyperref}
\usepackage[capitalise]{cleveref}
\usepackage[margin=1.25in]{geometry}

\usepackage{float}

\numberwithin{equation}{section}

\newtheorem{theorem}{Theorem}[section]
\newtheorem{proposition}[theorem]{Proposition}
\newtheorem{lemma}[theorem]{Lemma}

\newtheorem{corollary}[theorem]{Corollary}

\theoremstyle{definition}
\newtheorem{definition}[theorem]{Definition}

\newtheorem{remark}[theorem]{Remark}

\newcommand{\EE}{\mathbb E}
\newcommand{\RR}{\mathbb R}

\newcommand{\bX}{\mathbf X}

\DeclareMathOperator{\vol}{vol}
\DeclareMathOperator{\Pois}{Pois}

\newcommand{\eps}{\varepsilon}

\usepackage{etoolbox}
\AtBeginEnvironment{quote}{\par\small}

\title{Lower bounds for sphere packing in arbitrary norms}
\author{Carl Schildkraut}

\begin{document}

\begin{abstract}
We show that in any $d$-dimensional real normed space, unit balls can be packed with density at least
\[\frac{(1-o(1))d\log d}{2^{d+1}},\]
improving a result of Schmidt from 1958 by a logarithmic factor and generalizing the recent result of Campos, Jenssen, Michelen, and Sahasrabudhe in the $\ell_2$ norm. Our main tools are the graph-theoretic result used in the $\ell_2$ construction and volume bounds from convex geometry due to Petty and Schmuckenschl\"ager.
\end{abstract}

\maketitle

\section{Introduction}\label{sec:intro}

The question of how densely congruent spheres may be packed in $d$-dimensional space is a classical one. The ``trivial bound,'' i.e.~that at least a $2^{-d}$ proportion of space may be covered, was first improved by a factor of $2$ by Minkowski \cite{Minkowski1905} in 1905. The first asymptotic improvement (i.e.,~by more than a constant factor) on the trivial bound was due to Rogers \cite{Rogers47}, who gave a lower bound of the form $c(1-o(1))d2^{-d}$ for $c=2/e$. This constant was improved to $c=2$ in 1992 by Ball \cite{Ball} and to $c=65963$ by Venkatesh in 2012 \cite{Venkatesh}. Very recently, Campos, Jenssen, Michelen, and Sahasrabudhe \cite{CJMS} gave an asymptotic improvement, replacing the constant $c$ by $\frac12\log d$.\footnote{Here and throughout, $\log$ denotes the natural logarithm.}

\vspace{2mm}

Nothing about the question in the previous paragraph is specific to the Euclidean norm. Endowing $\RR^d$ with an arbitrary norm $\lVert\cdot\rVert$, one may attempt to pack as densely as possible copies of the unit ball $B_{\lVert\cdot\rVert}(0,1)=\{x\in\RR^d : \|x\|\leq 1\}$. Lower bounds of the form $c(1-o(1))d2^{-d}$ are known in this setting as well, with the first such bound given by Schmidt \cite{Schmidt58} in 1958. The constant $c$ was improved first by Rogers \cite{Rogers58}, and then to $c=\frac12\log 2$ by Schmidt \cite{Schmidt63}. In contrast to the $\ell_2$ problem, the constant in this more general setting has not been improved since 1963. In particular, the improvements due to Ball and Venkatesh in the $\ell_2$ case crucially use the symmetric structure of the $\ell_2$ sphere; they do not work naturally in more generality.

\vspace{2mm}

Our main result extends the result of \cite{CJMS} to arbitrary norms, with the same asymptotic improvement and the same constant factor. This improves upon the previous best-known bound by a logarithmic factor. We will phrase our results in terms of centrally symmetric convex bodies (i.e.,~centrally symmetric compact convex sets with nonempty interior), as our arguments are more natural in this language.\footnote{Convex bodies symmetric about $0$ and norms on $\RR^d$ are in bijection: the unit ball $\{x\in\RR^d : \|x\|\leq 1\}$ has nonempty interior, is centrally symmetric, is compact, and is convex, while if $K$ has nonempty interior and is centrally symmetric, compact, and convex, then $\|x\|=\inf\{\lambda\geq 0 : x\in\lambda K\}$ is a norm on $\RR^d$.} Given a compact set $K\subset\RR^d$ with nonempty interior, let $\delta_T(K)$ be the maximum density of a packing of translates of $K$ in $\RR^d$. 

\begin{theorem}\label{thm:main} For each compact, centrally symmetric convex body $K\subset\RR^d$,
\[\delta_T(K)\geq (1-o(1))\frac{d\log d}{2^{d+1}},\]
where the $o(1)$ term tends to zero with the dimension $d$, irrespective of the particular body $K$.    
\end{theorem}

We quickly mention the question of finding upper bounds. In the $\ell_2$ case, the best upper bound is of the shape $2^{-(0.599\ldots-o(1))d}$, due to Kabatjanski\u{\i} and Leven\v{s}te\u{\i}n \cite{KL} with constant-factor improvements by Cohn and Zhao \cite{CZ} and Sardari and Zargar \cite{SZ}. Some bodies $K$ (such as the hypercube) can of course be packed much more tightly, so there is no general-purpose upper bound in the setting of \cref{thm:main}. We are not aware of any centrally symmetric convex body $K\subset\RR^d$ whose translational packing density is known to be strictly less than the best known upper bound for the $\ell_2$ unit ball.

\vspace{2mm}

It is also interesting to consider translational packings of shapes which are not centrally symmetric. Via a result of Rogers and Shephard \cite{RS}, we may the following from \cref{thm:main}. 

\begin{corollary}\label{cor:main} For each compact, convex body $K\subset\RR^d$,
\[\delta_T(K)\geq (1-o(1))\frac{\sqrt\pi d^{3/2}\log d}{2^{2d+1}}.\]
\end{corollary}

We now make some comments about our approach. Historically, most of the arguments achieving dense packings of convex shapes have constructed \emph{lattice} packings, i.e.~packings in which the set of centers forms a lattice in $\RR^d$. In contrast, the packing given in \cite{CJMS} is \emph{amorphous}, lacking any global structure.  In the case of the $\ell_2$ unit ball, densities asymptotic to $cd2^{-d}$ were first achieved for amorphous packings in 2004 by Krivelevich, Litsyn, and Vardy \cite{KLV}, with constant-factor improvements by Jenssen, Joos, and Perkins \cite{JJP} and Fern\'andez, Kim, Liu, and Pikhurko \cite{FKLP}. For $\ell_p$ unit balls and some generalizations (\emph{superballs}), Xie and Ge \cite{XG} recently extended the techniques of \cite{JJP} to give lower bounds of the form $cd2^{-d}$. To our knowledge, ours is the first work achieving densities asymptotically greater than $2^{-d}$ for other convex bodies using amorphous packings.

\vspace{2mm}

In \cite{KLV}, the analysis of the amorphous packing is via tools from graph theory. Specifically, the authors use a result of Ajtai, Koml\'os, and Szemer\'edi \cite{AKS} that a graph with few triangles must have a large independent set. The improvement of Campos, Jenssen, Michelen, and Sahasrabudhe comes from a novel and more refined graph-theoretic result, which allows them to obtain an extra factor of $\log d$ in the size of the independent set. The construction we present in proving \cref{thm:main} follows their approach, using the graph-theoretic result of \cite{CJMS} to obtain the same logarithmic improvement.

\vspace{2mm}

To prove \cref{thm:main}, we will first randomly construct a discrete set $X$ of points from which to choose the centers of our copies of $K$. We will then use the graph-theoretic result of \cite{CJMS} to show that a large subset of $X$ may be chosen to be centers of a packing. To satisfy the preconditions of this graph-theoretic result, we will need to ensure that the copies of $K$ centered at points in $X$ do not have very large pairwise intersection. This requires some understanding of the geometry of $K$. We achieve this understanding by considering the \emph{polar projection body} of $K$, a centrally symmetric convex body associated to $K$ (see \cref{def:bodies}). A result of Schmuckenschl\"ager \cite{Schmuckenschlager} reduces our problem to studying the volume of the polar projection body; to do this, we appeal to Petty's projection inequality \cite{Petty}. This argument is carried out in \cref{sec:I-bound}.

\vspace{2mm}

In an earlier version of this article, the geometric argument proceeded instead via a reduction to the Bourgain slicing problem. From this reduction, we attained the main result by appealing to recent work \cite{Chen,KlartagLehec,JLV,Klartag23,Guan,KlartagLehec2} on the slicing problem. This approach is sketched in \cref{rmk:slicing}.

\subsection{Special cases}\label{sec:special}

Some particular choices of $K$ beside the $\ell_2$ unit ball warrant special mention.

\subsubsection*{\texorpdfstring{$\ell_p$}{lp}-balls}\label{sec:lp} Outside of the Euclidean norm, the most natural and well-studied norms on $\RR^d$ are the $\ell_p$ norms for $1\leq p\leq \infty$. For each fixed $p>2$, packings of density exponentially greater than $2^{-d}$ (i.e.~of density at least $(2-\eps_p)^{-d}$ for some $\eps_p>0$) were constructed in 1991 by Elkies, Odlyzko and Rush \cite{EOR}, and improved constructions were presented by Liu and Xing \cite{LX}. In contrast, no such bound is known for any $p\in [1,2]$. Recently, Xie and Ge \cite{XG} gave two new proofs of lower bounds of the form $cd2^{-d}$ for sphere packing in $\ell_p$ spaces with $1<p\leq 2$ (and some generalizations thereof) with techniques from statistical physics, based on the argument of Jenssen, Joos, and Perkins \cite{JJP} in the $\ell_2$ case. \cref{thm:main} improves this bound by a logarithmic factor.

\subsubsection*{Regular simplex}\label{sec:simplex} The regular $d$-simplex $\triangle_d\subset\RR^d$ is an attractive example of a convex shape about which to ask packing problems, as one of the only regular polytopes in large dimensions (along with the cross-polytope and the cube, which are $\ell_p$-balls for $p=1$ and $p=\infty$, respectively). Simplices do not pack translationally very tightly. Rogers and Shephard \cite{RS} give bounds of the form
\begin{equation}\label{eq:RS}
\frac{c_1d^{1/2}}{2^{2d}}\leq \delta_T(\triangle_d)\leq \frac{c_2d^{1/2}}{2^d}
\end{equation}
for constants $c_1,c_2$. Using Schmidt's bound for packing centrally symmetric shapes, the exponent $d^{1/2}$ in the lower bound of \eqref{eq:RS} can be easily improved to $d^{3/2}$; \cref{cor:main} improves this bound by a logarithmic factor.

\begin{remark}\label{rmk:simplex} The upper bound in \eqref{eq:RS} may be improved (for large $d$) to $2^{-c_3d}$ for some $c_3>1.428$. Let $\triangle_d$ be the $d$-simplex in $\RR^{d+1}$ of side length $\sqrt 2$ formed as the convex hull of the standard basis vectors. The key observation is that the symmetric convex body $\frac{\triangle_d-\triangle_d}2$ is exactly the intersection of the $(d+1)$-dimensional cross-polytope, i.e.~the $\ell_1$ unit ball $B_{\ell^1}^{d+1}$ in $\RR^{d+1}$, with the hyperplane $H$ described by $x_1+\cdots+x_{d+1}=0$. We sketch the remainder of the argument: if a set $A\subset H$ serves as the set of centers for a packing of $\frac{\triangle_d-\triangle_d}2$, then $A+\frac2{d+1}(1,1,\ldots,1)$ forms a set of centers for a packing of $B_{\ell^1}^{d+1}$. This gives the inequality
\begin{equation}\label{eq:crosspoly-relation}
\delta_T\left(B_{\ell^1}^{d+1}\right)\geq\frac{\vol(B_{\ell^1}^{d+1})}{\vol\big(\frac{\triangle_d-\triangle_d}2\big)}\delta_T\big(\tfrac{\triangle_d-\triangle_d}2\big)\frac{\sqrt{d+1}}2\geq\frac c{\sqrt d}\delta_T\big(\tfrac{\triangle_d-\triangle_d}2\big)
\end{equation}
for some constant $c>0$ (the $\frac12\sqrt{d+1}$ factor comes from the distinction between the sets of centers for the two convex bodies). A result of Rankin \cite{Rankin} gives an upper bound $\delta_T(B_{\ell^1}^{d+1})\leq 2^{-(0.428\ldots)d}$; combining this with \eqref{eq:crosspoly-relation} gives a similar upper bound on $\delta_T(\frac{\triangle_d-\triangle_d}2)$. This bound, in turn, yields an upper bound of $2^{-(1.428\ldots)d}$ on $\delta_T(\triangle_d)$ using the argument of Rogers and Shephard (see \cref{sec:non-sym}). (We thank Dmitrii Zakharov for bringing the connection between these two convex bodies to our attention.)
\end{remark}

\section{Outline and proof sketch}\label{sec:outline}

In this section, we outline our construction and state the main ingredients necessary for the proof of \cref{thm:main}. We give a proof of \cref{thm:main} assuming these intermediate results. Finally, we deduce \cref{cor:main} from \cref{thm:main}. 

\subsection{Construction of the packing}\label{sec:constr} As in \cite{CJMS}, the packings we construct will be ``amorphous,'' with fairly little structure. Following the argument in \cite{CJMS}, we construct the packing in the following way, depending on a parameter $\Delta$:
\begin{enumerate}
    \item Choose a set $\bX$ of points via a Poisson point process with intensity $2^{-d}\Delta$. Construct a graph $G$ whose vertex set is $\bX$ and whose edges are pairs $(x,y)$ for which $x+K$ and $y+K$ intersect. (Vertices in this graph should have degree approximately $\Delta$.)

    \item Remove from $\bX$ those vertices which have degree much larger than $\Delta$, or which have codegree more than $d^{-9}\Delta$ with some other vertex of $\bX$. Call the resulting graph $G'$.

    \item Find a large independent set $A\subset X$ in $G'$. The points corresponding to the vertices in $A$ form the centers of the copies of $K$ in our packing.
\end{enumerate}

\noindent The resulting packing will have density approximately $2^{-d}\log\Delta$, improving on the previous bound $cd2^{-d}$ as long as $\Delta$ can be taken to be super-exponential in $d$.

\subsection{Structure of the proof}\label{sec:structure} Let $K\subset\RR^d$ be a centrally symmetric convex body of volume $1$. Define the set
\[I_K=\left\{x\in\RR^d : \vol(K\cap (K+x))>d^{-10}\right\},\]
and select $\Delta_K=(d\vol(I_K))^{-1}$. Given a finite set of points $X$, let $G(X,K)$ be the graph with vertex set $X$ and with an edge $xy$ if and only if $x-y\in 2K$, i.e.~if and only if the translates $x+K$ and $y+K$ intersect. The first and second steps in the above outline are accomplished via the following lemma.

\begin{lemma}[cf.~{\cite[Lemma~2.1]{CJMS}}]\label{lem:graph-construction} Suppose $d>10$ and $\Delta$ are so that $d^{12}<\Delta\leq\Delta_K$. Let $\Omega\subset\RR^d$ be bounded and measurable. There exists some finite set of points $X\subset\Omega$ such that
\begin{itemize}
    \item $|X|\geq(1-2/d)(\Delta/2^d)\vol(\Omega)$,
    \item the graph $G(X,K)$ has maximum degree at most $\Delta+\Delta^{2/3}$, and
    \item the graph $G(X,K)$ has maximum codegree at most $d^{-9}\Delta$.
\end{itemize}
\end{lemma}

\noindent (Here, the \emph{codegree} of a pair of distinct vertices $(v_1,v_2)$ is the number of vertices $w$ of the graph for which both $v_1w$ and $v_2w$ are edges.) We prove \cref{lem:graph-construction} in \cref{sec:set-select}, following the above outline. To accomplish the third step of the construction, we use the following graph-theoretic result.

\begin{theorem}[{\cite[Theorem~1.3]{CJMS}}]\label{thm:CJMS} Let $G$ be a graph on $n$ vertices with maximum degree at most $\Delta$, and suppose that the maximum codegree of any two vertices in $G$ is at most $\Delta/(2\log\Delta)^7$. Then there exists an independent set in $G$ of size at least
\[(1-o(1))\frac{n\log\Delta}\Delta,\]
where the $o(1)$ tends to zero as $\Delta\to\infty$.
\end{theorem}

To complete the proof of \cref{thm:main}, we must lower-bound $\Delta_K$; equivalently, we must give an upper bound on the volume of $I_K$. Indeed, the following holds.

\begin{proposition}\label{prop:I-bound} There exists a constant $C_{\mathrm{int}}>0$ for which, for any centrally symmetric convex body $K\subset\RR^d$ of volume $1$,
\[\vol(I_K)\leq \left(\frac{C_{\mathrm{int}}\log^2d}{d}\right)^{d/2}.\]
\end{proposition}

\noindent It is in proving \cref{prop:I-bound} that we will use the volume bounds of Schmuckenschl\"ager and of Petty.

\begin{proof}[Proof of \cref{thm:main}] For each compact $\Omega\subset\RR^d$, we find a finite set $A\subset\Omega$ so that
\[|A|\geq(1-o(1))\frac{d\log d}{2^{d+1}}\vol(\Omega)\]
and the balls $x+K$ and $y+K$ are disjoint for $x,y\in A$. Via a standard compactness argument (see \cite[Section~1.1]{BMP}), having such a packing within each compact set is enough to give the desired bound on $\delta_T(K)$.

\vspace{2mm}

We may assume $d$ is sufficiently large: for $d$ small, we can use the trivial bound $\delta_T(K)\geq 2^{-d}$ and absorb the loss into the $o(1)$ term. Let $\Delta=\min\big(\Delta_K,d^{d^{8/7}}\big)$, so that (using \cref{prop:I-bound}) we have $\Delta_K\geq \Delta>d^{12}$. So, by \cref{lem:graph-construction}, we may choose a set $X\subset\Omega$ with
\[|X|\geq \frac{(1-o(1))\Delta\vol(\Omega)}{2^d}\]
so that $G(X,K)$ has maximum degree at most $\Delta+\Delta^{2/3}$ and maximum codegree at most $d^{-9}\Delta$. Since $\Delta\leq d^{d^{8/7}}$, we have $d^{-9}\Delta\leq \Delta/(2\log\Delta)^7$. So, we may apply \cref{thm:CJMS} to the graph $G(X,K)$ to find an independent set $A\subset X$ with
\[|A|\geq (1-o(1))\frac{|X|\log\Delta}{\Delta}\geq \frac{(1-o(1))\vol(\Omega)\log\Delta}{2^d}.\]
This set $A$ is our set of centers. Finally, \cref{prop:I-bound} implies that
\begin{align*}
\log\Delta
&\geq\min\left(d^{8/7}\log d,\log\Delta_K\right)\\
&=\min\left(d^{8/7}\log d,\log\left((d\vol I_K)^{-1}\right)\right)\\
&\geq \frac{(1-o(1))d\log d}2.
\end{align*}
This finishes the proof of \cref{thm:main}.
\end{proof}

\subsection{Non-symmetric sets}\label{sec:non-sym}

Here, we explain how to deduce \cref{cor:main} from \cref{thm:main}. This procedure goes back to Rogers and Shephard \cite{RS}, who coupled their bound on the volume of $K-K$ with an observation of Minkowski \cite{Minkowski1904} to attain the bound
\[\delta_T(K)\geq(1-o(1))\frac{\sqrt{\pi d}}{2^{2d}}\]
for any convex body $K$. For completeness, we outline the deduction.

\begin{lemma}[{\cite{Minkowski1904}}]\label{lem:non-sym} For any convex body $K\subset\RR^d$,
\[\frac{\delta_T(K)}{\vol K}=\frac{\delta_T\left(\frac{K-K}2\right)}{\vol\left(\frac{K-K}2\right)}.\]
\end{lemma}
\begin{proof}[Proof sketch] This equality follows from the fact that a discrete set $A\subset\RR^d$ serves as the centers for a packing of $K$ if and only if it serves as the centers for a packing of $\frac{K-K}2$.
\end{proof}

\begin{lemma}[Rogers--Shephard inequality; {\cite{RS}}]\label{lem:RS} For any convex body $K\subset\RR^d$,
\[\vol(K-K)\leq \binom{2d}d\vol K,\]
with equality if and only if $K$ is a simplex.
\end{lemma}

\noindent We remark that, since equality in \cref{lem:RS} holds in the case of a simplex $K=\triangle_d$, the trivial bound $\delta_T(\frac{\triangle_d-\triangle_d}2)\leq 1$ furnishes the upper bound $\delta_T(\triangle_d)\leq 2^d/\binom{2d}d\sim 2^{-d}\sqrt{\pi d}$ mentioned in the introduction.

\begin{proof}[Proof of \cref{cor:main}] Note that $L:=\frac{K-K}2$ is convex and centrally symmetric. So, chaining \cref{lem:non-sym}, \cref{lem:RS}, and \cref{thm:main} gives
\[\delta_T(K)=\frac{\vol K}{\vol L}\delta_T(L)\leq \frac{2^d}{\binom{2d}d}\delta_T(L)\leq \frac{(1-o(1))d\log d}{2\binom{2d}d}.\]
The corollary then follows from the asymptotic $\binom{2d}d=2^{2d}(\pi d)^{-1/2}(1+o(1))$.
\end{proof}

\section{Selecting the set \texorpdfstring{$X$}{X}: Proof of \texorpdfstring{\cref{lem:graph-construction}}{Lemma 2.1}}\label{sec:set-select}

We now prove \cref{lem:graph-construction}. As previously mentioned, our strategy will be to first pick a set of points $\bX\subset\Omega$ via a Poisson point process and then to remove a $o(1)$-fraction of the points to form $X$. The proof is quite similar to the proof of \cite[Lemma 2.1]{CJMS}. The additional complexity is in bounding the maximum codegree, and this is where the choice of $\Delta_K$ as determined by $\vol(I_K)$ is relevant. We will need the following standard tail bound for Poisson random variables.

\begin{lemma}[Follows from {\cite[Theorem 2.1]{JLR}}]\label{lem:poisson-conc} A random variable $Z\sim\Pois(\lambda)$ satisfies the following upper tail bound: for each $t\geq 1$,
\[\Pr\big[Z>(1+t)\lambda\big]\leq e^{-\lambda t/3}.\]
\end{lemma}

We now show \cref{lem:graph-construction}.

\begin{proof}[Proof of \cref{lem:graph-construction}] Through what follows, we write $I=I_K$, as the relevant body $K$ will not change. We choose the set $X$ in the following manner. Let $\bX$ be a set of points chosen according to a Poisson point process inside $\Omega$ with intensity $\lambda:=2^{-d}\Delta$. Define the following subsets $\bX_1,\bX_2,\bX_3\subset\bX$:
\begin{align*}
\bX_1&=\left\{x\in \bX : |\bX\cap(x+2K)|>\Delta+\Delta^{2/3}\right\},\\
\bX_2&=\left\{x\in\bX : \bX\cap (x+2I)\text{ contains a point other than $x$}\right\},\\
S_3&=\left\{(x,y)\in\bX^2 : x-y\not\in 2I\text{ and }|(\bX\setminus\{x,y\})\cap (x+2K)\cap (y+2K)|\geq d^{-9}\Delta\right\},\\
\bX_3&=\left\{x\in\bX : \exists y\in\bX\text{ with }(x,y)\in S_3\right\}.
\end{align*}
Set $X=\bX\setminus(\bX_1\cup\bX_2\cup\bX_3)$. We make the following observations.
\begin{itemize}
    \item By the definition of $\bX_1$, the degree of any vertex of $G(X,K)$ is at most $\Delta+\Delta^{2/3}$.
    \item By the definition of $\bX_2$, each pair $(x,y)$ of distinct elements of $X$ satisfies $x-y\not\in 2I$.
    \item By the definition of $\bX_3$ and the above property, the maximum codegree of $G(X,K)$ is at most $d^{-9}\Delta$.    
\end{itemize}
So, as long as $|X|\geq (1-2/d)(\Delta/2^d)\vol(\Omega)$, $X$ will satisfy the conditions of the lemma. We will in fact show that $\EE|X|$ exceeds this quantity; the result follows by Markov's inequality. We have
\begin{align*}
\EE|\bX|&=\lambda\vol(\Omega)=\Delta/2^d\cdot\vol(\Omega)\text{ and }\\
\EE|X|&\geq\EE|\bX|-\EE|\bX_1|-\EE|\bX_2|-\EE|S_3|.
\end{align*}
Hence, it suffices to show the following bounds:
\begin{align}
\EE|\bX_1|&\leq e^{-d}\EE|\bX|\label{eq:X1-bound}\\
\EE|\bX_2|&\leq d^{-1}\EE|\bX|\label{eq:X2-bound}\\
\EE|S_3|&\leq e^{-d}\EE|\bX|.\label{eq:S3-bound}
\end{align}

\begin{proof}[Proof of \eqref{eq:X1-bound}] Conditioned on $x\in\bX$, $|\bX\cap (x+2K))|-1$ is Poisson with mean
\[\lambda\vol\big((x+2K)\cap\Omega\big)\leq2^d\lambda=\Delta.\]
So, \cref{lem:poisson-conc} implies
\[\Pr[x\in\bX_1|x\in\bX]\leq e^{-(\Delta^{2/3}-1)/3}.\]
This bound, together with the fact that $\Delta>d^{12}$, gives \eqref{eq:X1-bound}.\phantom\qedhere
\end{proof}

\begin{proof}[Proof of \eqref{eq:X2-bound}] Conditioned on $x\in\bX$, the random variable $|\bX\cap (x+2I)|-1$ is Poisson with mean 
\[\lambda\vol\big((x+2I)\cap\Omega\big)\leq\lambda\vol(2I)=\Delta\vol(2I)2^{-d}=\Delta\vol(I)\leq\frac1d\]
So, $|\bX\cap(x+2I)|>1$, i.e.~$x\in\bX_2$, with probability at most $1-e^{-1/d}\leq 1/d$. This gives \eqref{eq:X2-bound}.\phantom\qedhere
\end{proof}

\begin{proof}[Proof of \eqref{eq:S3-bound}] If $(x,y)\in S_3$, then $x+2K$ and $y+2K$ must intersect, so $x-y\in 4K$. We now show that, for any pair of distinct points $(x,y)$ with $x-y\in 4K$,
\[\Pr\left[(x,y)\in S_3\big|x,y\in\bX\right]\leq e^{-d^{-10}\Delta}.\]
Indeed, fixing $x$ and $y$ with $x-y\not\in 2I$ and conditioning on $x,y\in\bX$, we have
\[\big|(\bX\setminus\{x,y\})\cap (x+2K)\cap (y+2K)\big|\sim\Pois\big(\lambda\vol((x+2K)\cap(y+2K))\big),\]
and
\[\vol\big((x+2K)\cap(y+2K)\big)=2^d\vol\left(K\cap\left(K+\frac{x-y}2\right)\right)\leq 2^dd^{-10}\]
since $(x-y)/2\not\in I$. We conclude that $|(\bX\setminus\{x,y\})\cap (x+2K)\cap (y+2K)|$ is Poisson with mean at most $d^{-10}\Delta$. As a result, the probability that it exceeds $d^{-9}\Delta$ is at most, by \cref{lem:poisson-conc},
\[\Pr\left[\Pois(d^{-10}\Delta)\geq 4d^{-10}\Delta\right]\leq e^{-d^{-10}\Delta},\]
as desired. Now, we can use this bound to upper-bound $\EE|S_3|$:
\begin{align*}
\EE|S_3|
&\leq e^{-d^{-10}\Delta}\cdot\EE\big|\{(x,y) : x,y\in\bX,\ x\neq y,\ x-y\in 4K\}\big|\\
&\leq e^{-d^{-10}\Delta}\cdot\EE|\bX|\cdot\sup_{x\in\Omega}\EE\left[\big|\{y\in\bX\setminus\{x\} : x-y\in 4K\}\big||x\in\bX\right]\\
&=e^{-d^{-10}\Delta}\cdot\EE|\bX|\cdot\lambda\vol(4K)=2^d\Delta e^{-d^{-10}\Delta}\cdot\EE|\bX|.
\end{align*}
The fact that $\Delta>d^{12}$ is enough to give \eqref{eq:S3-bound}. \phantom\qedhere
\end{proof}
\vspace{-7.7mm}
\end{proof}

\section{Bounding the volume of \texorpdfstring{$I_K$}{IK}: Proof of \texorpdfstring{\cref{prop:I-bound}}{Proposition 2.3}}\label{sec:I-bound}

We now prove \cref{prop:I-bound}, i.e.~that there exists a constant $C_{\mathrm{int}}>0$ so that, for every dimension $d$ and every centrally symmetric convex body $K$
\[\vol(I_K)\leq\left(\frac{C_{\mathrm{int}}\log^2d}{d}\right)^{d/2},\]
where as before
\[I_K=\left\{x\in\RR^d : \vol(K\cap (K+x))>d^{-10}\right\}.\]
To do this, we must define some auxiliary convex bodies. We will define these by their \emph{support functions}: given a convex body $K\subset\RR^d$, its support function is the function $h_K\colon\RR^d\to\RR$ defined by
\[h_K(y)=\sup_{x\in K}x\cdot y.\]

\begin{definition}\label{def:bodies} Given a convex body $K\subset\RR^d$, the \emph{polar body} $K^*$ of $K$ is given by
\[K^*:=\{y\in\RR^d : |h_K(y)|\leq 1\}.\]
The \emph{projection body} $\Pi K$ of $K$ is defined so that, for each $x$ with $\|x\|_2=1$,
\[h_{\Pi K}(x)=\vol_{\RR^{d-1}}(\text{projection of $K$ onto }x^\perp)\]
The \emph{polar projection body}, denoted $\Pi^*K$, is the polar of the projection body.    
\end{definition}

Our control on $I_K$ comes from the following result of Schmuckenschl\"ager:

\begin{lemma}[{\cite[Theorem~1]{Schmuckenschlager}}]\label{lem:bound-by-polarproj} For any symmetric convex body $K$ and any $\delta\in(0,1)$, one has
\[(1-\delta)\Pi^*K\subset\{x\in\RR^d : \vol(K\cap(K+x))>\delta\}\subset\log\left(\frac1\delta\right)\Pi^*K.\]
\end{lemma}

\noindent We sketch a proof of the latter containment, which is the only part we need, using the theory of logarithmically concave (``log-concave'') distributions. (For the relevant definitions, see the introduction of \cite{Prekopa}.) Our proof is essentially a streamlined version of Schmuckenschl\"ager's original proof.

\begin{proof}[Proof sketch] Let $f(x)=\vol(K\cap(K+x))$. We have that $f$ is the probability density function of the convolution of the uniform measure on $K$ with itself. The uniform measure on $K$ is log-concave, as $K$ is convex; therefore, results of Pr\'ekopa \cite[Theorems~2~and~7]{Prekopa} implies that $f$ is a log-concave function. Fix some vector $y\in\RR^d$ with $\|y\|_2=1$, and define $g\colon\RR_{\geq 0}\to\RR$ by $g(t)=f(ty)$. Since $f$ is log-concave, $g$ is as well. Let $\pi$ be the projection operator onto $y^\perp$. We claim that $g'(0)=-\vol(\pi(K))$. Indeed, this follows from the expression
\[g(t)=\int_{\pi(K)}\max\left((\text{length of }\pi^{-1}(z))-t,0\right)dz.\]
The fact that $g$ is log-concave thus implies
\[\log\vol(K\cap (K+ty))=\log g(t)\leq \log g(0)+t\frac{g'(0)}{g(0)}=-t\cdot h_{\Pi K}(y)=-h_{\Pi K}(ty).\]
Setting $x=ty$, we conclude that, if $\vol(K\cap (K+x))>\delta$, then $\log(1/\delta)\geq h_{\Pi K}(x)$. This is equivalent to $x\in(\log(1/\delta))\Pi^*K$.
\end{proof}

\noindent The second containment in \cref{lem:bound-by-polarproj} bounds $I_K$ by $\Pi^*K$ within a logarithmic factor in the dimension. It remains to upper-bound the volume of $\Pi^*K$. This can be done via an inequality of Petty.

\begin{lemma}[{\cite{Petty}}]\label{lem:petty} Let $K\subset\RR^d$ be a convex body, and let $B$ be a Euclidean ball satisfying $\vol K=\vol B$. Then
\[\vol\Pi^*K\leq\vol\Pi^*B.\]
\end{lemma}

\noindent In fact, much more can be said: the volume of $\Pi^*K$ is tightly controlled by the minimal surface area of a linear image of $K$; see \cite[Theorem~3.6]{GP}. Petty's inequality is enough for our purposes.

\begin{proof}[Proof of \cref{prop:I-bound}] For each dimension $k\geq 1$, write $\gamma_k=\pi^{k/2}/(\frac k2)!$ for the volume of the $k$-dimensional unit ball.\footnote{For nonnegative real $x$, we use $x!$ as a shorthand for $\Gamma(x+1)$, where $\Gamma$ denotes the gamma function.} Let $r_k$ be the radius of a unit-volume Euclidean ball in $\RR^k$, so that $\gamma_kr_k^k=1$. Finally, let $B$ be the unit-volume Euclidean ball in $\RR^d$. By \cref{lem:bound-by-polarproj,lem:petty}, we have
\begin{align}
\vol(I_K)
&=\vol\big(\{x\in\RR^d:\vol(K\cap(K+x))>d^{-10}\big)\notag\\
&\leq(10\log d)^d\vol(\Pi^*K)\notag\\
&\leq(10\log d)^d\vol(\Pi^*B)\label{eq:IK-first-bound}.
\end{align}
So, it remains to compute $\vol(\Pi^*B)$. For any $x\in\RR^d$ of Euclidean norm $1$, the projection of $B$ onto $x^\perp$ is a $(d-1)$-dimensional ball of radius $r_d$, and so
\[h_{\Pi B}(x)=\gamma_{d-1}r_d^{d-1}.\]
As a result, $\Pi B$ is a Euclidean ball of radius $\gamma_{d-1}r_d^{d-1}$, and $\Pi^*B$ is a Euclidean ball of radius $(\gamma_{d-1}r_d^{d-1})^{-1}$. We conclude
\begin{equation}\label{eq:polarproj-bound}
\vol(\Pi^*B)=\frac{\gamma_d}{(\gamma_{d-1}r_d^{d-1})^d}=\frac{\gamma_d^d}{\gamma_{d-1}^d}=\left(\frac{\sqrt\pi\left(\frac{d-1}2\right)!}{\left(\frac d2\right)!}\right)^d\leq\left(\frac{2\pi}d\right)^{d/2},
\end{equation}
where we have used that $\gamma_{d-1}r_{d-1}^{d-1}=\gamma_dr_d^d=1$ and that
\[\frac{x!}{\left(x-\frac12\right)!}\geq\sqrt x\]
for every real $x\geq 0$, which follows from the log-convexity of the gamma function and the property $x!/(x-1)!=x$.
The result follows for $C_{\mathrm{int}}=200\pi$ by combining \eqref{eq:IK-first-bound} and \eqref{eq:polarproj-bound}.
\end{proof}

\begin{remark}\label{rmk:slicing} Instead of bounding $I_K$ inside a scaled image of $\Pi^*K$, one may obtain a cruder estimate by bounding $I_K$ inside a Euclidean ball, once a suitable linear transformation has been applied to $K$ (to obtain an \emph{isotropic} body).\footnote{As mentioned in the introduction, this was the argument given in an earlier version of the article.} With this strategy, one may get a bound of the form
\begin{equation}\label{eq:slicing-bound}
\vol(I_K)^{1/d}\lesssim L_K\cdot\log d\cdot \vol\big(B_{\|\cdot\|_2}(0,1)\big)^{1/d}\lesssim\frac{\log d\cdot L_K}{\sqrt d},
\end{equation}
where $L_K$ is the \emph{slicing constant} of $K$, the reciprocal of the largest volume of intersection of $K$ with any hyperplane. In our proof of \cref{prop:I-bound}, we were able to upper-bound all relevant quantities by the corresponding quantities for a Euclidean ball. However, $L_K$ is (nearly) minimized, rather than maximized, when $K$ is a Euclidean ball (see \cite[pg.~203--206]{SlicingSurvey}). The question of giving constant upper bound on $L_K$ is the \emph{Bourgain slicing problem} \cite{Bourgain86,Bourgain87,BallThesis}, which was recently resolved by Klartag and Lehec \cite{KlartagLehec2}, building on recent works \cite{Chen,KlartagLehec,JLV,Klartag23,Guan}. Combining this result with \eqref{eq:slicing-bound} is enough to show \cref{prop:I-bound} and thus \cref{thm:main}. (In fact, any bound of the form $L_K\lesssim \log^{O(1)}d$ is enough to prove \cref{thm:main}, albeit with an $o(1)$ term which is larger by a constant factor.)

We note that the Bourgain slicing problem is closely related to other questions in convex geometry, such as the Kannan--Lov\'asz--Simonovits conjecture \cite{KLS} and the thin-shell problem of Anttila, Ball, and Perissinaki \cite{ABP} . More examples of equivalent problems are listed in \cite[Section~5]{MP}, and more detail on the historical context of and progress on the slicing problem is given in the survey article \cite{SlicingSurvey}.
\end{remark}

\section*{Acknowledgements}

The author would like to thank Jacob Fox for many helpful conversations and suggestions, Assaf Naor for communicating this proof of \cref{prop:I-bound}, and Dmitrii Zakharov for mentioning the connection between packing the simplex and the cross-polytope.

\bibliographystyle{alpha}
\bibliography{sphere-packing-arbitrary-norms.bib}

\end{document}